\documentclass{article}
\usepackage{arxiv}

\usepackage{amsmath}
\usepackage{amssymb}
\usepackage{amsthm}
\usepackage{cite}

\newtheorem{theorem}{Theorem}
\newtheorem{lemma}{Lemma}
\newtheorem{definition}{Definition}

\DeclareMathOperator*{\argmin}{arg\,min}

\newcommand{\interior}[1]{%
  {\kern0pt#1}^{\mathrm{o}}%
}

\begin{document}

\title{Accelerated First-Order Methods: Differential Equations and Lyapunov Functions}
\author{Jonathan W. Siegel \\
  Department of Mathematics\\
  Pennsylvania State University\\
  University Park, PA 16802 \\
  \texttt{jus1949@psu.edu}}

\maketitle

\begin{abstract}
    We develop a theory of accelerated first-order optimization from the viewpoint of differential equations and Lyapunov functions.
    Building upon the previous work of many researchers, we consider differential equations
    which model the behavior of accelerated gradient descent. Our main contributions are to provide a general framework for discretizating
    the differential equations to produce accelerated methods, and to provide physical intuition which helps explain the optimal damping rate. An important novelty is the generality of our approach, which leads to a unified derivation of a wide variety of methods, 
    including versions of Nesterov's accelerated gradient descent, FISTA, and accelerated coordinate descent.
\end{abstract}

\section{Introduction}
Minimizing convex and strongly convex functions is a fundamental problem which arises in many areas of science. We concern ourselves here with the problem
\begin{equation}\label{problem}
    \argmin_{x\in \mathbb{R}^d} f(x),
\end{equation}
where $f:\mathbb{R}^d\rightarrow \mathbb{R}$ is a strongly convex function. Due to the importance of very large scale problems of the form \eqref{problem}, which arise in machine learning
and data science, first-order methods have gained popularity in recent years. 
In practice, methods which only utilize gradient information are often the only ones which
can be applied to large scale problems of the form \eqref{problem}.

Motivated by this, there has been a lot of research into developing optimal first-order methods for convex optimization. Beginning with Polyak's discovery of the heavy ball method \cite{polyak1964some}, which attains an accelerated convergence rate locally, and the
seminal discovery of Nesterov's globally accelerated gradient descent \cite{nesterov1983method}, many different accelerated methods have been developed
by many authors. For instance, accelerated methods for solving composite optimization problems are developed in \cite{beck2009fast} and 
\cite{nesterov2007gradient} and an accelerated version of coordinate descent was developed in \cite{nesterov2012efficiency}, to name only
a few.

In spite of this progress, these methods have remained somewhat mysterious and difficult to understand. 
Consequently, there has been a lot of work in 
explaining these methods. For instance, in \cite{bubeck2015geometric}, a geometric explanation of acceleration is given, in \cite{allen2014linear} accelerated methods are derived as a coupling of gradient and mirror descent, and in \cite{su2014differential,wibisono2016variational,attouch2017asymptotic,attouch2020first} these methods are studied via the differential equations which they discretize. In 
\cite{wilson2016lyapunov,wibisono2016variational},
a detailed Lyapunov analysis of both the continuous and discrete dynamics of accelerated methods are presented and connected via
a discretization analysis. In addition, the Lyapunov analysis is shown to be equivalent to the technique of estimate sequences. In the non-smooth case, accelerated forward-backward methods such as FISTA \cite{beck2009fast} are analyzed in terms of the differential inclusions which they discretize in \cite{attouch2018fast,attouch2018convergence}. These lines of analysis are further extended in the recent work \cite{attouch2020first}.

In this paper, we provide more analysis of accelerated methods and their connection with continuous dynamics. We build upon the work in \cite{wibisono2016variational,su2014differential,wilson2016lyapunov,attouch2018fast,attouch2018convergence} and study accelerated first-order methods from the viewpoint of the underlying differential equations in both the smooth and non-smooth cases. We are mainly interested in the strongly convex case, but we also analyze accelerated methods for non-smooth objectives in the convex case in Section \ref{convex-section}.

Our main contribution is the derivation of a general framework for discretizing
the differential equations to produce accelerated methods. A framework for doing this has already been developed in \cite{wilson2016lyapunov},
and we attempt to build upon their work. Our contribution is to connect accelerated forward-backward methods for non-smooth problems
to differential equations and to provide a Lyapunov analysis for such methods. In addition, we also
connect our framework to accelerated coordinate methods \cite{nesterov2012efficiency,nesterov2017efficiency}, 
for example accelerated Gauss-Seidel \cite{tu2017breaking}.
The question of whether such methods, for example accelerated coordinate descent,
can systematically be connected to differential equations was posed in \cite{wibisono2016variational}.
Our treatment leads to a unified derivation of a wide variety of methods in the
literature, including deterministic methods such as Nesterov's 
accelerated gradient descent \cite{nesterov1983method} and FISTA \cite{beck2009fast}, and stochastic methods like accelerated coordinate descent.

The key to our theory is that many of these methods can be derived as special discretizations of damped Hamiltonian dynamics with an appropriately chosen damping rate.
The discretizations we consider all look very similar; they consist of an explicit forward step in position, a semi-implicit step in
velocity, and finally a small perturbation (second order in the step size) which ensures a sufficient decrease in the objective. This general
framework, which we introduce in sections three through six, provides a unified derivation of a wide variety of accelerated first-order methods.
We find it fascinating that so many methods can be obtained as discretizations of the same equations in such a simple way.

The paper is organized as follows. In Section \ref{differential-equation-section}, we briefly analyze the differential equations underlying accelerated methods, 
following the treatment in \cite{wilson2016lyapunov,wibisono2016variational}. In Section \ref{physical-intuition-section}, we analyze the linear case and give some physical intuition which explains the optimal damping rate and helps to explain why acceleration is possible. Then, in Section \ref{original_smooth}, we discuss how to discretize the differential equations to obtain accelerated methods when the objective is smooth, obtaining a variant
of Nesterov's accelerated gradient descent method. This analysis is then modified in Section \ref{non-smooth-section} to incorporate non-smooth composite objectives. In Section \ref{stochastic-section}, we further modify the theory
to incorporate stochastic methods such as randomized coordinate descent. In Sections \ref{non-smooth-section} and \ref{stochastic-section}, we also show how accelerated methods for
composite optimization and a version of accelerated coordinate descent follow as special cases of our theory. Then, in Section \ref{convex-section} we treat the convex case as well. Finally, we
provide some concluding remarks and further research directions.

\section{The Differential Equations}\label{differential-equation-section}
In this section, we present the differential equations underlying accelerated first-order optimization methods. The theory in the 
convex, i.e. sublinear, case was considered in \cite{su2014differential,wibisono2016variational}.
We are mainly concerned with the strongly convex case, which was studied in \cite{wilson2016lyapunov,wibisono2016variational},
and we briefly summarize some of their analysis in this section.

We first recall the notion of strong convexity.
\begin{definition}
Let $\alpha > 0$. A convex function $f$ is $\alpha$-strongly convex if for all $x,y$ and $g\in \partial f(y)$, it holds that
\begin{equation}
    f(x) \geq f(y) + \langle g, x-y \rangle + \frac{\alpha}{2}\|x - y\|^2.
\end{equation}
\end{definition}
\subsection{Strongly Convex Dynamics}
If $f$ is $\alpha$-strongly convex and differentiable, we consider the following damped Hamiltonian dynamics 
with potential energy $f$ (see \cite{wilson2016lyapunov}, equation 7)
\begin{equation} \label{damped_hamiltonian_dynamics}
    \dot{x} = v,~\dot{v} = -2\sqrt{\alpha} v - \nabla f(x).
\end{equation}
For non-differentiable $f$, we replace the gradient by an element of the sub-differential to obtain the dynamics
\begin{equation}\label{damped_hamiltonian_dynamics_non_smooth}
    \dot{x} = v,~- 2\sqrt{\alpha} v - \dot{v}\in \partial f(x).
\end{equation}

Following the argument in \cite{wilson2016lyapunov}, we use a Lyapunov function to prove that the objective error decreases at a linear rate of $-\sqrt{\alpha}$ under this dynamics.
\begin{theorem}\label{damped_dynamics_convergence}
Let $f$ be $\alpha$-strongly convex and differentiable. Assume that $x(t)$ and $v(t)$ obey the dynamics \ref{damped_hamiltonian_dynamics_non_smooth} (or equivalently \ref{damped_hamiltonian_dynamics}) and $v(0) = 0$. Then we have
\begin{equation}
    f(x(t)) - f(x^*) \leq 2e^{-\sqrt{\alpha}t}(f(x(0)) - f(x^*)),
\end{equation}
 where $x^*$ minimizes $f$.
\end{theorem}

Before we give the proof of Theorem \ref{damped_dynamics_convergence}, we remark upon the special damping rate of $2\sqrt{\alpha}$ which is taken in \eqref{damped_hamiltonian_dynamics_non_smooth}. In many practical instances, the strong convexity parameter may not be known and it may be a difficult problem to determine the correct damping rate. In this case, methods with an adaptive damping rate have been developed \cite{o2015adaptive,nesterov2013gradient}. In addition, when the objective is only convex, but not strongly convex, the damping rate should be taken decreasing to $0$ at a rate of $O(\frac{1}{t})$ \cite{su2014differential}.

\begin{proof}[Proof of Theorem \ref{damped_dynamics_convergence}]
Consider the Lyapunov function $$L(t) = f(x_t) - f(x^*) + \frac{1}{2}\|\sqrt{\alpha}(x_t - x^*) + v_t\|^2.$$ Here we have written $x_t$ for $x(t)$ and $v_t$ for $v(t)$ to simplify notation.
We will show that $L^\prime(t) \leq -\sqrt{\alpha}L(t)$. This completes the proof since 
\begin{equation}
    f(x_t) - f(x^*) \leq L(t) \leq e^{-\sqrt{\alpha}t}L(0) = e^{-\sqrt{\alpha}t}(f(x_0) - f(x^*) + \frac{\alpha}{2}\|x_0 - x^*\|^2),
\end{equation}
and $\frac{\alpha}{2}\|x_0 - x^*\|^2 \leq f(x_0) - f(x^*)$ by the strong convexity of $f$, so that
\begin{equation}
    f(x_t) - f(x^*) \leq 2e^{-\sqrt{\alpha}t}(f(x(0)) - f(x^*)).
\end{equation}
Using the properies of the subdifferential, we bound $L^\prime(t)$ as follows:
\begin{equation}
    L^\prime(t) = \langle g_t, \dot{x}_t\rangle + \langle \sqrt{\alpha}\dot{x}_t + \dot{v}_t, \sqrt{\alpha}(x_t - x^*) + v_t\rangle,
\end{equation}
where $g_t := - 2\sqrt{\alpha} v_t - \dot{v}_t\in \partial f(x_t)$.
We now use the dynamics \ref{damped_hamiltonian_dynamics} to evaluate $\dot{x}_t$ and $\dot{v}_t$, we obtain
\begin{equation}
    L^\prime(t) = \langle g_t, v_t\rangle + \langle -\sqrt{\alpha}v_t - g_t, \sqrt{\alpha}(x_t - x^*) + v\rangle.
\end{equation}
Simplifying this, we see that 
\begin{equation}\label{line315}
    L^\prime(t) = -\sqrt{\alpha}\langle g_t, (x_t - x^*)\rangle - \alpha \langle v_t, (x_t - x^*)\rangle - \sqrt{\alpha}\langle v_t, v_t\rangle.
\end{equation}
It is here that we use strong convexity, namely 
$$\langle g_t, (x_t - x^*)\rangle \geq f(x_t) - f(x^*) + \frac{\alpha}{2}\|x_t - x^*\|^2,$$
for any $g_t\in \partial f(x_t)$.
Plugging this into \ref{line315} and simplifying the inner products we get
\begin{equation}
    L^\prime(t) \leq -\sqrt{\alpha}\left(f(x_t) - f(x^*) + \frac{1}{2}\|\sqrt{\alpha}(x_t - x^*) + v_t\|^2\right) - \frac{\sqrt{\alpha}}{2}\|v_t\|^2,
\end{equation}
which implies
\begin{equation}
    L^\prime(t) \leq -\sqrt{\alpha}L(t)
\end{equation}
as desired.
\end{proof}

\section{Intuition Behind the Dynamics}\label{physical-intuition-section}
The intuition behind the dynamics in \eqref{damped_hamiltonian_dynamics} comes from imagining a particle in a potential defined by the objective $f$. Without any damping (friction),
the particle will oscillate freely in this potential and the total energy $f(x) + \frac{1}{2}\|v\|^2$ is conserved. The damping term $-\gamma v$ in \eqref{damped_hamiltonian_dynamics}
causes the particle to lose energy so that it will eventually settle at the minimum of $f$.

We want to choose the damping rate $\gamma$ so that the particle will settle in the minimum energy configuration as fast as possible. The optimal damping rate can be 
determined in the case when the objective $f$ is quadratic, which provides intuition behind the damping rate in \eqref{damped_hamiltonian_dynamics}. Finally, we consider discretizing \eqref{damped_hamiltonian_dynamics} with a fixed step-size in the quadratic case to help provide intuition about why acceleration is possible.

So let $f(x) = \frac{1}{2}x^TAx$ be a positive definite quadratic objective. We diagonalize $A$ and write $x(t)$ in the basis of eigenvectors of $A$ as
\begin{equation}
    x(t) = \displaystyle\sum_k x_k(t)w_k,~v(t) = \displaystyle\sum_k v_k(t)w_k,
\end{equation}
where $w_1,...,w_d$ are the eigenvectors of $A$ with corresponding eigenvalues $0 < \lambda_1 \leq \cdots \leq \lambda_d$. The dynamics decouples across each of the eigendirections and we get
\begin{equation}\label{damped_harmonic_oscillator}
    \dot{x}_k(t) = v_k(t),~\dot{v}_k(t) = -\gamma v_k(t) - \lambda_kx_k(t),
\end{equation}
a damped harmonic oscillator for each eigendirection. 

Analyzing a damped harmonic oscillator is an undergraduate physics exercise. The
characteristic polynomial of \eqref{damped_harmonic_oscillator} is
\begin{equation}
    p(z) = z^2 + \gamma z + \lambda_k.
\end{equation}
The roots of this polynomial, $z = \frac{1}{2}(-\gamma \pm \sqrt{\gamma^2 - 4\lambda_k})$, determine the behavior of the harmonic oscillator. The qualitative behavior of the
system depends upon whether the characteristic polynomial has two real roots, a repeated real root, or two imaginary roots.

If $\gamma^2 - 4\lambda_k > 0$, the characteristic polynomial has two real roots. This is the over-damped regime, and the oscillator stays on the same side of equilibrium
throughout the dynamics. The decay rate of the harmonic oscillator is dominated by the largest root, $z = \frac{1}{2}(-\gamma \pm \sqrt{\gamma^2 - 4\lambda_k})$.

If $\gamma^2 - 4\lambda_k < 0$, the characteristic polynomial has two imaginary roots. This is the under-damped regime, because the oscillator swings back and forth,
losing energy in each oscillation. The damping is not strong enough to keep the oscillator on the same side of equilibrium. The decay rate in this regime is equal to the
real part of the roots, $\text{Re}(z) = -\frac{1}{2}\gamma$.

If $\gamma^2  -4\lambda_k = 0$, the characteristic polynomial has repeated real roots. This is the critical damping rate for harmonic oscillator. The oscillator
stays on the same side of equilibrium but decays toward equilibrium as fast as possible. The decay rate is equal to the root $z = -\frac{1}{2}\gamma$.

Fixing $\lambda_k$, we see that the fastest decay rate possible is $-\sqrt{\lambda_k}$, which occurs for a critically damped harmonic oscillator. This follows since
in the under damped regime the decay rate is $-\frac{1}{2}\gamma$, with $\gamma < 2\sqrt{\lambda_k}$, and in the over damped regime we have
\begin{equation}
    \gamma^2 - 4\lambda_k = (\gamma + 2\sqrt{\lambda_k})(\gamma - 2\sqrt{\lambda_k}) > (\gamma - 2\sqrt{\lambda_k})^2
\end{equation}
The last inequality occurs because $\gamma > 2\sqrt{\lambda_k}$. This means that $$\sqrt{\gamma^2 - 4\lambda_k} > \gamma - 2\sqrt{\lambda_k}$$ and so for the larger
real root $$z = \frac{1}{2}(-\gamma + \sqrt{\gamma^2 - 4\lambda_k}) > -\sqrt{\lambda_k}$$

Now, in the dynamics \eqref{damped_hamiltonian_dynamics}, we must choose a single damping rate $\gamma$ for all eigenvalues $\lambda_k$. In order to obtain the fastest
possible decay rate, we want to maximize the slowest decay rate among the $\lambda_k$. 

Notice that the fastest this decay rate can be is $-\sqrt{\lambda_1}$, since the
harmonic oscillator corresponding to the smallest eigenvalue cannot decay any faster. Moreover, this decay rate is achieved when $\gamma$ is chosen so that the
mode corresponding to $\lambda_1$ is critically damped, because then all other modes will be under damped and will also decay at the rate $-\frac{1}{2}\gamma$.

What about the step size required for a stable discretization of the dynamics? Just as we chose a single damping rate for all eigenmodes, we must choose a single step 
size for all of the eigenvalues $\lambda_k$. If we use an integrator whose region of stability contains the negative unit semicircle (such as Runge-Kutta 4 \cite{runge1895numerische,wanner1996solving}), then the discretization of the harmonic oscillator corresponding to $\lambda_k$ will be stable if
$$\frac{1}{\Delta t} \geq \max\{|z_1|,|z_2|\}$$
where $z_1$ and $z_2$ are the (possibly complex) roots of the characteristic equation. Since the optimal damping parameter is $2\sqrt{\alpha}$, all of the modes are
either critically damped or under damped. So the roots $z_i$ are complex and we calculate
$$|z_1| = |z_2| = \frac{1}{2}\sqrt{\gamma^2 + 4\lambda_k - \gamma^2} = \sqrt{\lambda_k} \leq \sqrt{\lambda_d}$$

So we see that the dynamics decays with exponential rate $-\sqrt{\lambda_1}$ and the step size required for a stable discretiziation is $\Delta t<1/\sqrt{\lambda_d}$. 
This explains why discretizing \eqref{damped_hamiltonian_dynamics} produces a method which requires $O(\sqrt{\kappa})$ iterations to converge, where $\kappa = \frac{\lambda_d}{\lambda_1}$ is the condition number. 

Indeed, for the quadratic problem we can use any integrator whose region of stability contains the negative unit semicircle to obtain
an accelerated method. Compare the resulting methods with the Chebyshev semi-iterative methods used for solving linear systems, see \cite{golub2013matrix}, Section 10.1.5 or \cite{varga1962iterative}, Chapter 5. The Chebyshev semi-iterative methods use (shifted, depending upon the eigenvalue bounds of $A$) Chebyshev polynomials, while the accelerated methods derived above use a different sequence of polynomials with the same asymptotics.

\section{Discrete Dynamics for Smooth Objectives}\label{original_smooth}
Theorem \ref{damped_dynamics_convergence} concerns the convergence of the dynamics \eqref{damped_hamiltonian_dynamics} in the continuous case. The existence of discrete schemes which are able to achieve accelerated convergence has been known since the work of Polyak and Nesterov \cite{polyak1964some,nesterov1983method}.
In this section, we connect these two viewpoints and derive a version of Nesterov's accelerated gradient descent \cite{nesterov1983method} as a particular
discretization of \eqref{damped_hamiltonian_dynamics}. The method is made up of a forward step in position ($x$), a semi-implicit
step in velocity ($v$), and finally a small perturbation (second order in the step size) which ensures a sufficient decrease in the 
Lyapunov function. This perturbation moves $x$ to decrease the objective and moves $v$ to compensate, so that the second part of 
the Lyapunov function (the squared norm) doesn't increase. These steps are all explained in detail in this section. 

We begin by briefly recalling the definition of 
smoothness.

\begin{definition}
Let $L > 0$. A differentiable function $f$ is $L$-smooth if
\begin{equation}
    \|\nabla f(x) - \nabla f(y)\| \leq L\|x - y\|.
\end{equation}
\end{definition}
It is an easy consequence of the above definition that
\begin{equation}
    f(y) \leq f(x) + \langle\nabla f(x), y-x\rangle + \frac{L}{2}\|x-y\|^2.
\end{equation}

Our goal will now be to discretize \eqref{damped_hamiltonian_dynamics} so that a discrete version of the Lyapunov function in the proof of 
Theorem \ref{damped_dynamics_convergence}, for instance
$$
    L_n = f(x_n) - f(x^*) + \frac{1}{2}\|\sqrt{\alpha}(x_n - x^*) + v_n\|^2,
$$
will be decreased by a constant factor in each timestep. In order to make the Lyapunov argument work, however, a slight modification is necessary. We will actually consider the discrete Lyapunov function
\begin{equation}\label{lyapunov_function_smooth}
    L_n = f(x_n) - f(x^*) + \frac{1}{2}\|\sqrt{\alpha}(x_n - x^*) + (1+s\sqrt{\alpha})v_n\|^2,
\end{equation}
where $s$ is the step size of the discretization.

We first note that the $L$-smoothness of $f$ implies that
$$f\left(x_n - \frac{1}{L}\nabla f(x_n)\right) - f(x^*) \leq f(x_n) - f(x^*) - \frac{1}{2L}\|\nabla f(x_n)\|^2,$$
i.e. that taking a small gradient step ensures a decrease of the objective. Observe also that the second part of the Lyapunov
function, $$\frac{1}{2}\|\sqrt{\alpha}(x_n - x^*) + (1+s\sqrt{\alpha})v_n\|^2,$$ can be kept constant by adjusting $v_n$ appropriately, which will guarantee a sufficient decrease of the entire Lyapunov function. Specifically, we adjust $v_n$ so that $\sqrt{\alpha}(x_n - x^*) + (1+s\sqrt{\alpha})v_n$ is kept constant. Thus, if we subtract $\frac{1}{L}\nabla f(x_n)$ from $x_n$, then we must add $(1+s\sqrt{\alpha})^{-1}\frac{\sqrt{\alpha}}{L}$ to $v_n$. Putting this
together, we obtain the following update:
\begin{equation}\label{decrease_update}
    x_n \gets x_n - \frac{1}{L}\nabla f(x_n),~v_n\gets v_n + (1+s\sqrt{\alpha})^{-1}\frac{\sqrt{\alpha}}{L}\nabla f(x_n),
\end{equation}
which decreases the Lyapunov function by $\frac{1}{2L}\|\nabla f(x_n)\|^2$, i.e. so that
$$L_n\gets L_n - \frac{1}{2L}\|\nabla f(x_n)\|^2.$$
We call this a sufficient decrease update and it plays an important role in ensuring the global convergence of our discretization,
which we now introduce.

Consider the following discretization of \eqref{damped_hamiltonian_dynamics}, which consists of a forward Euler step
in $x$ and a semi-implicit step in $v$, followed by a sufficient decrease update:
\begin{equation} \label{discretization_smooth_deterministic}
    \begin{split}
        x_n^\prime &= x_n + sv_n \\
        v_n^\prime &= v_n - (1+s\sqrt{\alpha})^{-1}\left(s\sqrt{\alpha}v_n + s\nabla f(x_n^\prime)\right) - s\sqrt{\alpha}v_n^\prime \\
        x_{n+1} &= x_n^\prime - \frac{1}{L}\nabla f(x_n^\prime) \\
        v_{n+1} &= v_n^\prime + (1+s\sqrt{\alpha})^{-1}\frac{\sqrt{\alpha}}{L} \nabla f(x_n^\prime).
    \end{split}
\end{equation}
The most complicated part of this discretization is the semi-implicit update for $v$. Unfortunately, we haven't found a simpler
update for which the discretization remains stable, i.e. for which global convergence can be proved.

We now show that as long as the step size $s \leq \frac{1}{\sqrt{L}}$, the discretization \eqref{discretization_smooth_deterministic} converges linearly. Note that it is critical that we are able to take the step size as large as $\frac{1}{\sqrt{L}}$. This results in an accelerated convergence rate of $(1+\kappa^{-\frac{1}{2}})^{-n}$.
\begin{theorem} \label{accelerated_gradient_descent}
    If $s \leq \frac{1}{\sqrt{L}}$ and $f$ is $\alpha$-strongly convex and $L$-smooth, then the iteration \eqref{discretization_smooth_deterministic} satisfies
    $$L_{n+1} \leq (1 + s\sqrt{\alpha})^{-1}L_n.$$
    In particular, if $v_0 = 0$, then
    $$f(x_n) - f(x^*) \leq 2\left(1 + s\sqrt{\alpha}\right)^{-n}(f(x_0) - f(x^*)).$$
\end{theorem}

Before proving Theorem \ref{accelerated_gradient_descent}, we compare \eqref{discretization_smooth_deterministic} with the traditional Nesterov accelerated gradient scheme \cite{nesterov1983method}, which is given by:
\begin{equation}
 \begin{split}
 &x_n = y_{n} - \frac{1}{L}\nabla f(y_n)\\
 &y_{n+1} = x_n + \left(\frac{\sqrt{L} - \sqrt{\alpha}}{\sqrt{L} + \sqrt{\alpha}}\right)(x_n - x_{n-1}),
 \end{split}
\end{equation}
and the heavy ball method \cite{polyak1964some}, given by:
\begin{equation}
  x_{n+1} = x_n - \frac{4}{(\sqrt{L} + \sqrt{\alpha})^2}\nabla f(x_n) + \left(\frac{\sqrt{L} - \sqrt{\alpha}}{\sqrt{L} + \sqrt{\alpha}}\right)^2(x_n - x_{n-1}).
\end{equation}
By setting $v_n = x_n - x_{n-1}$, both of these methods can be seen to be discretizations of the dynamics \eqref{damped_hamiltonian_dynamics}, with Nesterov's scheme corresponding to a step size of $\frac{1}{\sqrt{L}}$ and Polyak's method corresponding to a larger step size of $\frac{2}{\sqrt{L} + \sqrt{\alpha}}$. However, Polyak's larger step size comes at the cost of only local convergence and convergence for quadratic objectives, while Nesterov's scheme enjoys global convergence for general strongly convex objectives. 

In Theorem \ref{accelerated_gradient_descent}, we are able to take a step size which matches, but does not exceed, Nesterov's scheme. In addition, our convergence rate of $(1+\kappa^{-\frac{1}{2}})^{-n}$ is slightly worse than Nesterov's rate of $(1-\kappa^{-\frac{1}{2}})^{n}$. For this reason, we don't expect \eqref{discretization_smooth_deterministic} to perform better in practice than Nesterov's scheme. We merely present it as a connection between the continuous dynamics \eqref{damped_hamiltonian_dynamics} and discrete accelerated methods.

\begin{proof}[Proof of Theorem \ref{accelerated_gradient_descent}]
Throughout the proof, we will use the following elementary fact. Let $t_n$ be some quantity which changes throughout our 
iteration. Then
$$\frac{1}{2}\|t_{n+1}\|^2 - \frac{1}{2}\|t_n\|^2 = \langle t_{n+1} - t_n, t_n\rangle + \frac{1}{2}\|t_{n+1} - t_n\|^2$$
and
$$\frac{1}{2}\|t_{n+1}\|^2 - \frac{1}{2}\|t_n\|^2 = \langle t_{n+1} - t_n, t_{n+1}\rangle - \frac{1}{2}\|t_{n+1} - t_n\|^2.$$
Applying each of these identities once, we see that if $t_{n+1} - t_n = a + b$, then
$$\frac{1}{2}\|t_{n+1}\|^2 - \frac{1}{2}\|t_n\|^2 = \langle a, t_n\rangle + \langle b,t_{n+1}\rangle + \frac{1}{2}\|a\|^2 - \frac{1}{2}\|b\|^2.$$
We will use these identities without explicit mention in what follows. We now prove that
$$L_{n+1} \leq L_n - s\sqrt{\alpha}L_{n+1},$$
where $L_n$ is the Lyapunov function given in \eqref{lyapunov_function_smooth}.
To do so, we calculate the change in $L$ due to the forward step in $x$,
\begin{equation}\label{forward-step}
\begin{split}
    L(x_n^\prime, v_n) - L_n =~ &f(x_n^\prime) - f(x_n) + s\langle \sqrt{\alpha}v_n, \sqrt{\alpha}(x^\prime_n - x^*) + (1+s\sqrt{\alpha})v_n\rangle \\
    & - \frac{s^2\alpha}{2}\|v_n\|^2,
    \end{split}
\end{equation}
and the change due to the semi-implicit step for $v$, noting that the change in $v$ can be broken up as 
$$(1+s\sqrt{\alpha})(v_n^\prime - v_n) = -(s\sqrt{\alpha}v_n + s\nabla f(x_n^\prime)) - s\sqrt{\alpha}(1+s\sqrt{\alpha})v_n^\prime,$$
to get
\begin{equation}\label{backward-step}
    \begin{split}
        L(x_n^\prime, v_n^\prime) - L(x_n^\prime, v_n) = &-s\langle \sqrt{\alpha}v_n + \nabla f(x_n^\prime), \sqrt{\alpha}(x^\prime_n - x^*) + (1+s\sqrt{\alpha})v_n\rangle \\
& - s\langle \sqrt{\alpha}(1+s\sqrt{\alpha})v_n^\prime, \sqrt{\alpha}(x^\prime_n - x^*) + (1+s\sqrt{\alpha})v_n^\prime\rangle \\
& + \frac{s^2}{2}\|\sqrt{\alpha}v_n + \nabla f(x_n^\prime)\|^2 \\
& - \frac{s^2\alpha}{2}\|v_n^\prime\|^2.
    \end{split}
\end{equation}
Adding equations \eqref{forward-step} and \eqref{backward-step}, collecting terms, and recalling that $x_n^\prime - x_n = sv_n$, we obtain
\begin{equation}\label{line248}
    \begin{split}
         L(x_n^\prime, v_n^\prime) - L_n = &f(x_n^\prime) - f(x_n) - \langle\nabla f(x_n^\prime), x_n^\prime - x_n\rangle\\
         & - s\sqrt{\alpha}\langle\nabla f(x_n^\prime), x_n^\prime - x^*\rangle \\
         & - s\sqrt{\alpha}\langle (1+s\sqrt{\alpha})v_n^\prime, \sqrt{\alpha}(x^\prime_n - x^*) + (1+s\sqrt{\alpha})v_n^\prime\rangle \\
         & + \frac{s^2}{2}\|\nabla f(x_n^\prime)\|^2 - \frac{s^2\alpha}{2}\|v_n^\prime\|^2.
    \end{split}
\end{equation}
The terms on the first line are $\leq 0$, by the convexity of $f$. The inner product on the second line can be bounded using
the strong convexity of $f$, as
$$\langle\nabla f(x_n^\prime), x_n^\prime - x^*\rangle \geq f(x_n^\prime) - f(x^*) + \frac{\alpha}{2}\|x_n^\prime - x^*\|^2.$$
Plugging this bound into equation \eqref{line248} and completing the square with line three, yields
\begin{equation}
    \begin{split}
        L(x_n^\prime, v_n^\prime) - L_n \leq & -s\sqrt{\alpha}L(x_n^\prime, v_n^\prime) + \frac{s^2}{2}\|\nabla f(x_n^\prime)\|^2 \\
        & - \frac{s^2\alpha}{2}\|v_n^\prime\|^2 - \frac{s\sqrt{\alpha}}{2}(1+s\sqrt{\alpha})^2\|v_n^\prime\|^2.
    \end{split}
\end{equation}
The sufficient decrease update now decreases the Lyapunov function by at least $\frac{s^2}{2}\|\nabla f(x_n^\prime)\|^2$ since
$s\leq \frac{1}{\sqrt{L}}$, and we obtain
\begin{equation}
    L_{n+1} - L_n \leq -s\sqrt{\alpha}L(x_n^\prime, v_n^\prime).
\end{equation}
To complete the proof, we merely note that $L_{n+1} \leq L(x_n^\prime, v_n^\prime)$ (the sufficient decrease update decreased the
Lyapunov function). This implies
\begin{equation}
    L_{n+1} - L_n \leq -s\sqrt{\alpha}L_{n+1}
\end{equation}
as desired.

By induction, we thus have 
$$f(x_n) - f(x^*) \leq L_n \leq \left(1 + s\sqrt{\alpha}\right)^{-n}L_0.$$
Finally, if $v_0 = 0$, we get, since $f$ is $\alpha$-strongly convex,
$$L_0 = f(x_0) - f(x^*) + \frac{\alpha}{2}\|x_0 - x^*\|^2 \leq 2(f(x_0) - f(x^*)).$$
\end{proof}
\section{Discrete Dynamics for Non-smooth Objectives}\label{non-smooth-section}
In this section we consider the situation where the objective $f$ is not smooth, but we still assume that $f$ is $\alpha$-strongly
convex. We first consider simply replacing $\nabla f(x_n^\prime)$ by some element in the sub-differential of 
$f$, say $g_n\in \partial f(x_n^\prime)$. We observe that the
argument for smooth functions continues to apply as long as we can find an analog of the sufficient decrease update \eqref{decrease_update}, i.e. if we 
have an update
\begin{equation}
    x'_n \gets x'_n - \delta_n,~v_n\gets v_n + (1+s\sqrt{\alpha})^{-1}\sqrt{\alpha}\delta_n
\end{equation}
such that 
\begin{equation}\label{decrease-condition}
 f(x'_n - \delta_n) \leq f(x'_n) - \frac{s^2}{2}\|g_n\|^2.
\end{equation}

Unfortunately, in many cases of interest this is not possible. To get around this, we allow $g_n \notin \partial f(x_n^\prime)$, and generalize the condition \eqref{decrease-condition}.
In particular, we will show that our Lyapunov argument still works if we replace the decrease condition \eqref{decrease-condition} by the more general condition
\begin{equation} \label{stability_decrease_condition}
    \forall z,~f(x'_n - \delta_n) - f(z) \leq \langle g_n, x_n^\prime - z\rangle - \frac{\alpha}{2}\|x_n^\prime - z\|^2 - \frac{s^2}{2}\|g_n\|^2.
\end{equation}

Let us examine this condition for a moment. Note that if $g_n \in \partial f(x_n^\prime)$, then by the strong convexity, we have
$$\forall z,~f(x_n^\prime) - f(z) \leq \langle g_n, x_n^\prime - z\rangle - \frac{\alpha}{2}\|x_n^\prime - z\|^2.$$
So, in this case, the above condition is equivalent to a decrease in the objective.
\begin{equation}
    f(x'_n - \delta_n) = f(x_n^\prime - \delta_n) \leq f(x_n^\prime) - \frac{s^2}{2}\|g_n\|^2.
\end{equation}

So what the new condition \eqref{stability_decrease_condition} does is simply to allow $g_n\notin \partial f(x_n^\prime)$, but still to enforce
a combined decrease and strong convexity condition. We will see that for many problems of interest, in particular composite
optimization, $g_n$ and $\delta_n$ can be chosen to satisfy this condition. This will lead to an accelerated version of forward-backward
iteration, which is similar to FISTA \cite{beck2009fast} and the methods in \cite{attouch2018convergence}.

Inserting this into the scheme \eqref{discretization_smooth_deterministic}, we arrive at the following discretization.
\begin{equation} \label{discretization_non_smooth_deterministic}
    \begin{split}
        x_n^\prime &= x_n + sv_n \\
        v_n^\prime &= v_n - (1+s\sqrt{\alpha})^{-1}\left(s\sqrt{\alpha}v_n + sg_n\right) - s\sqrt{\alpha}v_n^\prime \\
        x_{n+1} &= x_n^\prime - \delta_n \\
        v_{n+1} &= v_n^\prime + (1+s\sqrt{\alpha})^{-1}\sqrt{\alpha} \delta_n,
    \end{split}
\end{equation}
where $g_n$ and $\delta_n$ are chosen so that \eqref{stability_decrease_condition} holds.

This is the same as in the smooth case, except that the gradients have been replaced by $g_n$ and the sufficient decrease update has been changed.
We now prove that this scheme leads to an accelerated method. The proof is very similar to the smooth case. In this case, we do not even need to
assume the strong convexity of $f$. This assumption is subsumed by \eqref{stability_decrease_condition}.
\begin{theorem}\label{accelerated_gradient_descent_non_smooth}
Assume that $g_n$ and $\delta_n$ are chosen so that the condition \eqref{stability_decrease_condition} holds at every iteration of 
the scheme \ref{discretization_non_smooth_deterministic}.
Then we have
$$L_{n+1} \leq (1 + s\sqrt{\alpha})^{-1}L_n,$$
where $L_n$ is the same Lyapunov function as in the smooth case (equation \eqref{lyapunov_function_smooth}).

In particular, if $v_0 = 0$, we get
    $$f(x_n) - f(x^*) \leq \left(1 + s\sqrt{\alpha}\right)^{-n}\left(f(x_0) - f(x^*) + \frac{\alpha}{2}\|x_0 - x^*\|^2\right).$$
\end{theorem}
\begin{proof}
We proceed exactly as in the proof of theorem \ref{accelerated_gradient_descent}, replacing $\nabla f(x_n^\prime)$ by $g_n$ to obtain,
in place of equation \eqref{line248},
\begin{equation}\label{line334}
    \begin{split}
         L(x_n^\prime, v_n^\prime) - L_n = &f(x_n^\prime) - f(x_n) - \langle g_n, x_n^\prime - x_n\rangle\\
         & - s\sqrt{\alpha}\langle g_n, x_n^\prime - x^*\rangle \\
         & - s\sqrt{\alpha}\langle (1+s\sqrt{\alpha})v_n^\prime, \sqrt{\alpha}(x^\prime_n - x^*) + (1+s\sqrt{\alpha})v_n^\prime\rangle \\
         & + \frac{s^2}{2}\|g_n\|^2 - \frac{s^2\alpha}{2}\|v_n^\prime\|^2.
    \end{split}
\end{equation}
Noting that by the construction of the sufficient decrease update (the last two lines of \eqref{discretization_non_smooth_deterministic}), $$L_{n+1} - L(x_n^\prime, v_n^\prime) = f(x_{n+1}) - f(x_n^\prime),$$
we get
\begin{equation}
    \begin{split}
        L_{n+1} - L_n = &f(x_{n+1}) - f(x_n) - \langle g_n, x_n^\prime - x_n\rangle\\
         & - s\sqrt{\alpha}\langle g_n, x_n^\prime - x^*\rangle \\
         & - s\sqrt{\alpha}\langle (1+s\sqrt{\alpha})v_n^\prime, \sqrt{\alpha}(x^\prime_n - x^*) + (1+s\sqrt{\alpha})v_n^\prime\rangle \\
         & + \frac{s^2}{2}\|g_n\|^2 - \frac{s^2\alpha}{2}\|v_n^\prime\|^2.
    \end{split}
\end{equation}
We now apply the decrease condition \eqref{stability_decrease_condition} with $z = x_n$ and $z = x^*$ to the first two lines of this equation.
This gives
\begin{equation}
    \begin{split}
        L_{n+1} - L_n \leq &-\frac{\alpha}{2}\|x_n^\prime - x_n\|^2 - \frac{s^2}{2}\|g_n\|^2\\
         & - s\sqrt{\alpha}\left(f(x_{n+1}) - f(x^*) + \frac{\alpha}{2}\|x_n^\prime - x^*\|^2 + \frac{s^2}{2}\|g_n\|^2\right) \\
         & - s\sqrt{\alpha}\langle (1+s\sqrt{\alpha})v_n^\prime, \sqrt{\alpha}(x^\prime_n - x^*) + (1+s\sqrt{\alpha})v_n^\prime\rangle \\
         & + \frac{s^2}{2}\|g_n\|^2 - \frac{s^2\alpha}{2}\|v_n^\prime\|^2.
    \end{split}
\end{equation}
Completing the square and noting that the sufficient decrease update (the last two lines of \eqref{discretization_non_smooth_deterministic}) is designed so that 
$$\frac{1}{2}\|\sqrt{\alpha}(x_n^\prime - x^*) + (1+s\sqrt{\alpha})v_n^\prime\|^2 = \frac{1}{2}\|\sqrt{\alpha}(x_{n+1} - x^*) + (1+s\sqrt{\alpha})v_{n+1}\|^2,$$
we see that
\begin{equation}
    L_{n+1} - L_n \leq -s\sqrt{\alpha}L_{n+1}
\end{equation}
as desired.

As in the proof of Theorem \ref{accelerated_gradient_descent}, by induction we have 
$$f(x_n) - f(x^*) \leq L_n \leq \left(1 + s\sqrt{\alpha}\right)^{-n}L_0.$$
Finally, if $v_0 = 0$, we get
$$L_0 = f(x_0) - f(x^*) + \frac{\alpha}{2}\|x_0 - x^*\|^2,$$
which proves the second statement.
\end{proof}

\subsection{Accelerated Forward-Backward Splitting}
In this subsection, we apply theorem \ref{accelerated_gradient_descent_non_smooth} to strongly convex composite objectives, 
i.e. objectives of the form
\begin{equation}
    f(x) = g(x) + h(x),
\end{equation}
where $g$ is $\alpha$-strongly convex and $L$-smooth, and $h$ is an arbitrary convex function. We also assume that we are
able to compute a proximal update for $h$, i.e. solve
\begin{equation}\label{proximal_update}
    y^* = \text{prox}_{s,h}(x) = \argmin_y h(y) + \frac{1}{2s}\|y-x\|^2
\end{equation}
The proximal update is essentially a step of backward Euler with step size $s$, hence the name accelerated forward-backward method.
For many convex functions of interest, the proximal update can be efficiently computed.

For example, if $h$ is the characteristic function
of a convex set $S$, then \eqref{proximal_update} is just a projection onto $S$. In this case, Theorem \ref{accelerated_gradient_descent_non_smooth} 
recovers a version of accelerated projected gradient descent for strongly convex objectives.

Another example of interest is $h(x) = \|x\|_1$, in which case \eqref{proximal_update} is just soft-thresholding with parameter $s$. In this instance
Theorem \ref{accelerated_gradient_descent_non_smooth} recovers a version of FISTA \cite{beck2009fast}, and more generally for other $h$ we obtain a version of accelerated forward-backward descent \cite{attouch2018convergence}, which are designed for strongly convex objectives.

Our goal is to show how $g_n$ and $\delta_n$ can be chosen to satisfy the decrease condition \eqref{stability_decrease_condition}. The next lemma
answers this question for us.
\begin{lemma}\label{forward-backward-lemma}
    Assume that $f(x) = g(x) + h(x)$, with $g$ an $L$-smooth, $\alpha$-strongly convex function and $h$ a convex function. Then setting
    $$g_n = \frac{1}{s^2}\left(x_n^\prime - \text{\normalfont prox}_{s^2,h}\left(x_n^\prime - s^2\nabla g(x_n^\prime)\right)\right)$$
    and $\delta_n = s^2g_n$ will satisfy the condition \eqref{stability_decrease_condition} as long as $s\leq \frac{1}{\sqrt{L}}$.
\end{lemma}
Again it is important that we are able to choose $s$ as large as $\frac{1}{\sqrt{L}}$, since this results in the accelerated rate of $(1+\kappa^{-\frac{1}{2}})^{-n}$.
\begin{proof}
Note that the choice of $g_n$ and $\delta_n$ implies that
\begin{equation}
    x_{n+1} = \argmin_x h(x) + \frac{1}{2s^2}\|x - \left(x_n^\prime - s^2\nabla g(x_n^\prime)\right)\|^2.
\end{equation}
    We denote by $y_n$ the intermediate point $x_n^\prime - s^2\nabla g(x_n^\prime)$ so that the above becomes
    \begin{equation}\label{line251}
        x_{n+1} = \argmin_x h(x) + \frac{1}{2s^2}\|x - y_n\|^2.
    \end{equation}
    Now let $z$ be arbitrary and note that our goal is to bound
    $$f(x_{n+1}) - f(z) = (g(x_{n+1}) - g(z)) + (h(x_{n+1}) - h(z)).$$
    We consider the second of these terms first. Note that equation \eqref{line251} implies that $L(y_n - x_{n+1}) \in \partial h(x_{n+1})$.
    This means that (since $h$ is convex)
    \begin{equation}\label{line258}
        h(x_{n+1}) - h(z) \leq \frac{1}{s^2}\langle (y_n - x_{n+1}), x_{n+1} - z\rangle.
    \end{equation}
    We proceed to bound
    \begin{equation}
        g(x_{n+1}) - g(z) = (g(x_{n+1}) - g(x_n^\prime)) + (g(x_n^\prime) - g(z)).
    \end{equation}
    The first term above is bounded due to the $L$-smoothness of $g$ and the assumption that $s^2\leq \frac{1}{L}$:
    $$g(x_{n+1}) - g(x_n^\prime) \leq \langle\nabla g(x_n^\prime), x_{n+1} - x_n^\prime\rangle + \frac{1}{2s^2}\|x_{n+1} - x_n^\prime\|^2.$$
    The second is bounded due to the strong convexity of $g$:
    $$g(x_n^\prime) - g(z) \leq \langle\nabla g(x_n^\prime), x_n^\prime - z\rangle - \frac{\alpha}{2}\|x_n^\prime - z\|^2.$$
    Combining these two bounds with equation \eqref{line258} and noting that 
    $$g_n = \frac{1}{s^2}(x^\prime_n - x_{n+1}) = \frac{1}{s^2}(y_n - x_{n+1}) + \nabla g(x_n^\prime),$$ 
    we obtain
    \begin{equation}
        f(x_{n+1}) - f(z) \leq \langle g_n, x_{n+1} - z\rangle - \frac{\alpha}{2}\|x_n^\prime - z\|^2 + \frac{s^2}{2}\|g_n\|^2.
    \end{equation}
    Now we write $$x_{n+1} - z = x_{n+1} - x_n^\prime + x_n^\prime - z = -s^2g_n + x_n^\prime - z,$$
    to get
    \begin{equation}
        f(x_{n+1}) - f(z) \leq \langle g_n, x_n^\prime - z\rangle - \frac{\alpha}{2}\|x_n^\prime - z\|^2 - \frac{s^2}{2}\|g_n\|^2,
    \end{equation}
    which is exactly \eqref{stability_decrease_condition}.
\end{proof}

\section{Discrete Stochastic Dynamics}\label{stochastic-section}
In this section, we extend our theory to what we call stochastic discretizations. By this we simply mean schemes which introduce randomness
in each iteration. These are not really discretizations of the dynamics \eqref{damped_hamiltonian_dynamics} in a strict sense of the term, but lead to a class of accelerated methods nonetheless. The important new step here is to modify the sufficient decrease update appropriately when the gradient is sampled randomly in a certain sense. To illustrate the ideas, we derive a variant of accelerated coordinate descent. 

We begin by considering smooth objectives $f$. Recall that for smooth objectives the sufficient decrease update was critical for obtaining a stable accelerated method. What we needed was
a way of decreasing the objective sufficiently by perturbing $x$. Then we could perturb $v$ appropriately to keep the second term in
our Lyapunov function constant. In this way, we obtained an update of the form \eqref{decrease_update}, which reduced the Lyapunov function
by at least $\frac{s^2}{2}\|\nabla f(x_n^\prime)\|^2$.

When deriving stochastic accelerated methods, we want to replace $\nabla f(x_n^\prime)$ by some sample $g_n$, where 
$\mathbb{E}_n(g_n) = \nabla f(x_n^\prime)$ (here $\mathbb{E}_n$ denotes the expectation taken with respect to randomness introduced
in iteration $n$, essentially a conditional expectation). 
It turns out that this will work as long as we can guarantee an objective decrease of at least
$\frac{s^2}{2}\|g_n\|^2$. Note here that there is no expectation inside of the norm, i.e. this is the norm of the actual gradient sample
encountered at iteration $n$.

We consider the following discretization:
\begin{equation} \label{discretization_smooth_stochastic}
    \begin{split}
        x_n^\prime &= x_n + sv_n \\
        v_n^\prime &= v_n - (1+s\sqrt{\alpha})^{-1}\left(s\sqrt{\alpha}v_n + sg_n\right) - s\sqrt{\alpha}v_n^\prime \\
        x_{n+1} &= x_n^\prime - \delta_n \\
        v_{n+1} &= v_n^\prime + (1+s\sqrt{\alpha})^{-1}\sqrt{\alpha} \delta_n,
    \end{split}
\end{equation}
where $\mathbb{E}_n(g_n) = \nabla f(x_n^\prime)$ and $\delta_n$ is chosen (dependent on $g_n$) so that
\begin{equation}\label{smooth_decrease_condition}
    f(x_{n+1}) \leq f(x_n^\prime) - \frac{s^2}{2}\|g_n\|^2.
\end{equation}
We show that this method will achieve an accelerated convergence rate under this conditions.
\begin{theorem}\label{accelerated_stochastic_smooth}
    Assume that $f$ is $\alpha$-strongly convex and differentiable. Then as long as condition \eqref{smooth_decrease_condition} holds, the iterates
    of \eqref{discretization_smooth_stochastic} will satisfy
    $$\mathbb{E}_n(L_{n+1}) \leq (1 + s\sqrt{\alpha})^{-1}L_n,$$
    where $L_n$ is the same Lyapunov function introduced in \eqref{lyapunov_function_smooth}.
    
    In particular, if $v_0 = 0$ we have
    $$\mathbb{E}(f(x_n) - f(x^*)) \leq 2(1 + s\sqrt{\alpha})^{-n}(f(x_0) - f(x^*)).$$
\end{theorem}
This theorem will follow as a special case of the theorem for non-smooth schemes, so we omit the proof for the moment.

We now turn to stochastic schemes for non-smooth functions. The idea is very similar to the deterministic scheme for 
non-smooth functions. We want to choose $g_n$ as a random sample of an element in the subgradient $\partial f(x_n^\prime)$,
however, we don't restrict $\mathbb{E}_n(g_n)\in \partial f(x_n^\prime)$. Instead, we enforce a stochastic version of 
the constraint \eqref{stability_decrease_condition}. The scheme we consider is same as \eqref{discretization_non_smooth_deterministic}
\begin{equation} \label{discretization_non_smooth_non_deterministic}
    \begin{split}
        x_n^\prime &= x_n + sv_n \\
        v_n^\prime &= v_n - (1+s\sqrt{\alpha})^{-1}\left(s\sqrt{\alpha}v_n + sg_n\right) - s\sqrt{\alpha}v_n^\prime \\
        x_{n+1} &= x_n^\prime - \delta_n \\
        v_{n+1} &= v_n^\prime + (1+s\sqrt{\alpha})^{-1}\sqrt{\alpha} \delta_n,
    \end{split}
\end{equation}
except that we allow $g_n$ and $\delta_n$ to be (potentially dependent) random variables and enforce the following condition
\begin{equation} \label{stability_decrease_condition_non_deterministic}
    \forall z,~f(x_{n+1}) - f(z) \leq \langle \mathbb{E}_n(g_n), x_n^\prime - z\rangle - \frac{\alpha}{2}\|x_n^\prime - z\|^2 - \frac{s^2}{2}\|g_n\|^2.
\end{equation}
which only differs from \eqref{stability_decrease_condition} in the expectation taken in the first inner product.

Note first that the schemes \eqref{discretization_smooth_stochastic} and \eqref{discretization_non_smooth_non_deterministic}
are exactly the same, except that the condition \eqref{stability_decrease_condition_non_deterministic} is weaker than the
conditions enforced in the smooth case. This follows since if $f$ is differentiable and $\alpha$-strongly
convex, then
$$\forall z,~f(x_n^\prime) - f(z) \leq \langle \nabla f(x_n^\prime), x_n^\prime - z\rangle - \frac{\alpha}{2}\|x_n^\prime - z\|^2.$$
This, combined with the requirement in \eqref{discretization_smooth_stochastic} that $\mathbb{E}_n(g_n) = \nabla f(x_n^\prime)$ and
that $\delta_n$ is chosen so that \eqref{smooth_decrease_condition} holds, implies the condition \eqref{stability_decrease_condition_non_deterministic}.
This means that the proof below will imply Theorem \ref{accelerated_stochastic_smooth}.
In fact, the scheme \eqref{discretization_non_smooth_non_deterministic} is most often applied in this way to smooth, strongly convex functions.
In this case it leads to different versions of accelerated coordinate descent, as we will show later.

We now prove that the scheme \eqref{discretization_non_smooth_non_deterministic} achieves the desired accelerated convergence rate.
As in the non-smooth deterministic case, we don't need to assume that our objective is strongly convex. This assumption is
superseded by condition \eqref{stability_decrease_condition_non_deterministic}.
\begin{theorem}\label{accelerated_gradient_descent_non_smooth_non_deterministic}
Assume that $g_n$ and $\delta_n$ are chosen so that the condition \eqref{stability_decrease_condition_non_deterministic} holds at every iteration of 
the scheme \eqref{discretization_non_smooth_non_deterministic}.
Then we will have
$$\mathbb{E}_n(L_{n+1}) \leq (1 + s\sqrt{\alpha})^{-1}L_n,$$
where $L_n$ is the same Lyapunov function as in the smooth case (equation \eqref{lyapunov_function_smooth}).

In particular, if $v_0 = 0$, we will have
$$\mathbb{E}(f(x_n) - f(x^*)) \leq \left(1 + s\sqrt{\alpha}\right)^{-n}\left(f(x_0) - f(x^*) + \frac{\alpha}{2}\|x_0 - x^*\|^2\right).$$
\end{theorem}
\begin{proof}[Proof of Theorem \ref{accelerated_gradient_descent_non_smooth_non_deterministic}]
As in the proof of Theorem \ref{accelerated_gradient_descent_non_smooth}, we obtain (equation \eqref{line334})
\begin{equation}\label{line545}
    \begin{split}
         L(x_n^\prime, v_n^\prime) - L_n = &f(x_n^\prime) - f(x_n) - \langle g_n, x_n^\prime - x_n\rangle\\
         & - s\sqrt{\alpha}\langle g_n, x_n^\prime - x^*\rangle \\
         & - s\sqrt{\alpha}\langle (1+s\sqrt{\alpha})v_n^\prime, \sqrt{\alpha}(x^\prime_n - x^*) + (1+s\sqrt{\alpha})v_n^\prime\rangle \\
         & + \frac{s^2}{2}\|g_n\|^2 - \frac{s^2\alpha}{2}\|v_n^\prime\|^2.
    \end{split}
\end{equation}
We note that the sufficient decrease update (the last two lines of \eqref{accelerated_gradient_descent_non_smooth_non_deterministic}) implies that $L_{n+1} - L(x_n^\prime, v_n^\prime) = f(x_{n+1}) - f(x_n^\prime)$, so that
\begin{equation}
    \begin{split}
         L_{n+1} - L_n = &f(x_{n+1}) - f(x_n) - \langle g_n, x_n^\prime - x_n\rangle\\
         & - s\sqrt{\alpha}\langle g_n, x_n^\prime - x^*\rangle \\
         & - s\sqrt{\alpha}\langle (1+s\sqrt{\alpha})v_n^\prime, \sqrt{\alpha}(x^\prime_n - x^*) + (1+s\sqrt{\alpha})v_n^\prime\rangle \\
         & + \frac{s^2}{2}\|g_n\|^2 - \frac{s^2\alpha}{2}\|v_n^\prime\|^2.
    \end{split}
\end{equation}
We now apply the decrease condition \eqref{stability_decrease_condition_non_deterministic} with $z = x_n$ and $z = x^*$ to the first two lines of this equation. Because we have an expectation $\mathbb{E}_n(g_n)$ in the condition \eqref{stability_decrease_condition_non_deterministic}, this results in
extra terms which are an inner produce with the difference $g_n - \mathbb{E}_n(g_n)$.
We get
\begin{equation}
    \begin{split}
        L_{n+1} - L_n \leq &-\frac{\alpha}{2}\|x_n^\prime - x_n\|^2 - \frac{s^2}{2}\|g_n\|^2 - \langle g_n - \mathbb{E}_n(g_n), x_n^\prime - x_n\rangle\\
         & - s\sqrt{\alpha}\left(f(x_{n+1}) - f(x^*) + \frac{\alpha}{2}\|x_n^\prime - x^*\|^2 + \frac{s^2}{2}\|g_n\|^2\right) \\
         & - s\sqrt{\alpha}\langle g_n - \mathbb{E}_n(g_n), x_n^\prime - x^*\rangle\\
         & - s\sqrt{\alpha}\langle (1+s\sqrt{\alpha})v_n^\prime, \sqrt{\alpha}(x^\prime_n - x^*) + (1+s\sqrt{\alpha})v_n^\prime\rangle \\
         & + \frac{s^2}{2}\|g_n\|^2 - \frac{s^2\alpha}{2}\|v_n^\prime\|^2.
    \end{split}
\end{equation}
Collecting terms and completing the square as in the previous proofs, we get (throwing out unnecessary negative terms)
\begin{equation}
    \begin{split}
        L_{n+1} - L_n \leq &-s\sqrt{\alpha}\left(f(x_{n+1} - f(x^*) + \frac{1}{2}\|\sqrt{\alpha}(x^\prime_n - x^*) + (1+s\sqrt{\alpha})v^\prime_n\|^2\right)\\
        & - \langle g_n - \mathbb{E}_n(g_n), x_n^\prime - x_n\rangle\\
        & - s\sqrt{\alpha}\langle g_n - \mathbb{E}_n(g_n), x_n^\prime - x^*\rangle.
    \end{split}
\end{equation}
Recalling that the sufficient decrease update (the last two lines of \eqref{accelerated_gradient_descent_non_smooth_non_deterministic}) was designed so that
$$\frac{1}{2}\|\sqrt{\alpha}(x^\prime_n - x^*) + (1+s\sqrt{\alpha})v^\prime_n\|^2 = \frac{1}{2}\|\sqrt{\alpha}(x_{n+1} - x^*) + (1+s\sqrt{\alpha})v_{n+1}\|^2,$$
we see that
\begin{equation}
    \begin{split}
        L_{n+1} - L_n \leq &-s\sqrt{\alpha}L_{n+1}\\
        & - \langle g_n - \mathbb{E}_n(g_n), x_n^\prime - x_n\rangle\\
        & - s\sqrt{\alpha}\langle g_n - \mathbb{E}_n(g_n), x_n^\prime - x^*\rangle.
    \end{split}
\end{equation}
Finally, we take the expectation $\mathbb{E}_n$ on both sides to obtain
    $$\mathbb{E}_n(L_{n+1}) - L_n \leq -s\sqrt{\alpha}\mathbb{E}_n(L_{n+1}).$$
which implies
\begin{equation}
    \mathbb{E}_n(L_{n+1}) \leq (1 + s\sqrt{\alpha})^{-1}L_n,
\end{equation}
as desired.

Taking the expectation with respect to the randomness introduced in previous iterations, we obtain
$$\mathbb{E}(L_{n+1}) \leq (1 + s\sqrt{\alpha})^{-1}\mathbb{E}(L_n),$$
so that
$$\mathbb{E}(f(x_n) - f(x^*)) \leq \mathbb{E}(L_{n+1}) \leq (1 + s\sqrt{\alpha})^{-n}L_0.$$
Finally, if $v_0 = 0$ we get
$$L_0 \leq f(x_0) - f(x^*) + \frac{\alpha}{2}\|x_0 - x^*\|^2,$$
which completes the proof of the second statement. Note that if $f$ is $\alpha$-strongly convex, then
$$L_0 \leq 2(f(x_0) - f(x^*)),$$
which applies in the case of Theorem \ref{accelerated_stochastic_smooth}.
\end{proof}
\subsection{Accelerated Coordinate Descent}
We now show how two versions of accelerated coordinate descent follow as special cases of Theorem
\ref{accelerated_stochastic_smooth}. 

The premise we consider is that the objective $f$ is $\alpha$-strongly convex and the gradient is coordinate-wise $L_i$-smooth, i.e.
if we denote by $\nabla f(x)_i$ the $i$-th coordinate of the gradient of $f$, then
$$|\nabla f(x)_i - \nabla f(x + ce_i)_i| \leq L_i|c|.$$
Intuitively, this means that the $i$-th diagonal entry of the Hessian of $f$ is bounded by $L_i$ at each point.

Note also that the $L_i$-smoothness implies that we obtain a sufficient decrease when moving in the direction $i$, i.e.
$$f\left(x - \frac{1}{L_i}\nabla f(x)_i\right) \leq f(x) - \frac{1}{2L_i}\|\nabla f(x)_i\|^2.$$
This will be exactly what we need when choosing $g_n$ and $\delta_n$ in the scheme \eqref{discretization_smooth_stochastic}.

To obtain accelerated coordinate descent, we choose a coordinate $i$ at random, with the 
probability of choosing coordinate $i$ proportional to $\sqrt{L_i}$. We then set 
$$g_n = \frac{1}{s\sqrt{L_i}} \nabla f(x_n^\prime)_i.$$
and choose $s$ so that $\mathbb{E}_n(g_n) = \nabla f(x_n^\prime)$. This gives a step size of
$$s = \left(\displaystyle\sum_{i = 1}^n \sqrt{L_i}\right)^{-1}.$$

We now note that the required decrease in \eqref{smooth_decrease_condition} is
$$f(x_{n+1}) \leq f(x_n^\prime) - \frac{s^2}{2}\|g_n\|^2 = f(x_n^\prime) - \|sg_n\|^2 = f(x_n^\prime) - \frac{1}{2L_i}\|\nabla f(x)_i\|^2.$$

So we must simply choose $\delta_n$ so that
\begin{equation}
    f(x_n^\prime - \delta_n) \leq f(x_n^\prime) - \frac{1}{2L_i}\|\nabla f(x)_i\|^2.
\end{equation}
By the $L_i$-smoothness, the choice $\delta_n = \frac{1}{L_i}\nabla f(x)_i$ will work. This recovers a version of the original
accelerated coordinate descent in \cite{nesterov2012efficiency}. The relationship between the new method and Nesterov's original method is similar to the relationship in the deterministic case described in section \ref{original_smooth}.

Another possible choice is to choose $\delta_n = \frac{1}{L_j}\nabla f(x)_j$ where $j$ is the coordinate which maximizes
the decrease $\frac{1}{2L_j}\|\nabla f(x)_j\|^2$. This recovers the accelerated semi-greedy scheme presented in \cite{lu2018accelerating}.

Finally, we would like to conclude by explaining why accelerated coordinate descent is very efficient if $\nabla f(x)_i$ can be
calculated using only a small number (say $k \ll n$) of entries $x_{i_1},...,x_{i_k}$.

The issue is that each step of accelerated coordinate descent requires updating the full vectors $x$ and $v$. Namely in the
first and second steps in \eqref{discretization_smooth_stochastic} we update the whole vectors, while in the last two steps we only update the
selected coordinate $i$ ($g_n$ and $\delta_n$ are only non-zero in one coordinate).

The key observation is that we actually don't need to update all of the $x$ and $v$ coordinates in each step. To see why, imagine for a moment
that $g_n = \delta_n = 0$ in each step. Then the iteration becomes
\begin{equation}
    \begin{split}
        x_{n+1} &= x_n + sv_n \\
        v_{n+1} &= (1+s\sqrt{\alpha})^{-2}v_n, \\
    \end{split}
\end{equation}
which can be solved in closed form any number of iterations in the future. The key is to treat all of the coordinates in this manner
between their selection times. When coordinate $i$ is selected, this simple iteration must be modified to include
$\nabla f(x)_i$. However, until it is selected again, we can evaluate $x_i$ and $v_i$ at any iteration in closed form, as long as
we know how many iterations have passed and what the values were the last time $i$ was selected. This observation permits an efficient
implementation of \eqref{discretization_smooth_stochastic} as long as $\nabla f(x)_i$ doesn't depend on too many indices $i_1,...,i_k$.
This idea can be used to construct fast solvers for sparse, symmetric, diagonally dominant linear systems; compare with the work in
\cite{lee2013efficient}, for instance.

\section{Accelerated Methods in the Convex Case}\label{convex-section}
In the rest of the article, we were mainly interested in the case of strongly convex objectives $f$. In this section, we treat the convex case as well. So throughout this section, we assume that $f$ is convex and $L$-smooth, but not necessarily strongly convex. In this case, the differential equations modelling accelerated gradient descent were first determined and analyzed in \cite{su2014differential}. The equations introduced there are
\begin{equation}
 \label{variable_damped_hamiltonian_dynamics}
    \dot{x} = v,~\dot{v} = -\frac{r}{t}v - \nabla f(x),
\end{equation}
with $r \geq 3$. So for objectives which are not strongly convex, the damping rate isn't fixed, but should decay at a rate of $O(\frac{1}{n})$. With the help of the Lyapunov function
\begin{equation}\label{convex-lyapunov}
 L(t) = t^2(f(x(t)) - f(x^*)) + \frac{1}{2}\|(r-1)(x(t) - x^*)+tv(t)\|^2,
\end{equation}
it is shown in \cite{su2014differential} that the dynamics \eqref{variable_damped_hamiltonian_dynamics} with $v(0) = 0$ satisfies
\begin{equation}
 f(x(t)) - f(x^*) \leq \frac{(r-1)^2}{2t^2}\|x(0) - x^*\|^2.
\end{equation}
The discrete scheme analyzed in \cite{su2014differential} in this case is given by
\begin{equation}\label{eq-814}
 x_0 = y_0,~x_{n+1} = y_n - \frac{1}{L}\nabla f(y_n),~y_{n+1} = x_{n+1} + \frac{n}{n+r}(x_{n+1} - x_n).
\end{equation}
Inspired by this, in \cite{siegel2018accelerated}, the following somewhat more general scheme is analyzed:
\begin{equation}\label{general-convex-accelerated-scheme}
 x_0 = y_0,~x_{n+1} = y_n - \gamma_n\nabla f(y_n),~y_{n+1} = x_{n+1} + \alpha_n(x_{n+1} - x_n),
\end{equation}
where the following convergence theorem is proved.
\begin{theorem}[Theorem 3.3.2 in \cite{siegel2018accelerated}]\label{thesis-theorem}
 Let $q_n$ be a sequence of non-negative real number satisfying $q_0 = 0$ and 
 \begin{equation}\label{q-condition}
 (q_{n+1} + 1)^2 \leq (q_n + 2)^2 + 1.
 \end{equation}
 Then, if we set $\alpha_n = \frac{q_n}{2 + q_{n+1}}$ in \eqref{general-convex-accelerated-scheme} and choose $\gamma_n$ to be decreasing and satisfying the descent condition
 
\begin{equation}\label{descent-condition-convex}
 f(x_{n+1}) \leq f(y_n) - \frac{\gamma_n}{2}\|\nabla f(y_n)\|^2,
\end{equation}
 then we have
 \begin{equation}
  f(x_n) - f(x^*) \leq \frac{2}{\gamma_nq_n(q_n+2)}\|x_0 - x^*\|^2.
 \end{equation}

\end{theorem}
Note that in order to apply Theorem \ref{thesis-theorem} we do not even need to know the smoothness parameter $L$. We simply need to ensure that step $\gamma_n$ is chosen small enough to satisfy \eqref{descent-condition-convex} (and is decreasing). If $f$ is $L$-smooth, $\gamma_n = \frac{1}{L}$, and $q_n = \frac{2n}{r-1}$, we obtain the scheme \eqref{eq-814} from \cite{su2014differential}. We can obtain a slightly faster convergence rate by enforcing equality in \eqref{q-condition}.

Our purpose in this section is to extend this result to the non-smooth setting, by modifying the descent condition \eqref{descent-condition-convex}.

Similar to the approach in Section \ref{non-smooth-section}, the key to the analysis in the non-smooth case is to replace the descent condition \eqref{descent-condition-convex} by the condition
\begin{equation}\label{non-smooth-convex-descent-condition}
 \forall z,~f(x_{n+1}) - f(z) \leq \langle g_n, y_n - z\rangle - \frac{\gamma_n}{2}\|g_n\|^2,
\end{equation}
where $x_{n+1} = y_n - \gamma_ng_n$ for an appropriately chosen $g_n$ and $\gamma_n$. Note that if $f$ is $L$-smooth and we set $g_n = \nabla f(x_n)$ and $\gamma_n\leq \frac{1}{L}$, then this condition follows from \eqref{descent-condition-convex} since $f(y_n) - f(z) \leq \langle \nabla f(y_n), y_n - z\rangle$ by the convexity of $y$.

The most important instance where the condition \eqref{non-smooth-convex-descent-condition} can be guaranteed for a non-smooth function $f$ is when the objective $f$ is of the form
\begin{equation}
 f(x) = g(x) + h(x),
\end{equation}
where $g$ is $L$-smooth and convex, and $h$ is a potentially non-smooth function for which we can calculate the proximal step in equation \eqref{proximal_update}. In this case, updating $x_{n+1}$ via a forward-backward step
\begin{equation}
 x_{n+1} = \text{prox}_{\gamma_n,h}(y_n - \gamma_n\nabla g(y_n)),
\end{equation}
enables us to ensure that condition \eqref{non-smooth-convex-descent-condition} is satisfied for $\gamma_n \leq \frac{1}{L}$ (here we set $g_n = \frac{1}{\gamma_n}(y_n - x_{n+1})$ so that $x_{n+1} = y_n - \gamma_ng_n$). In particular, we have the following analog of Lemma \ref{forward-backward-lemma}.
\begin{lemma}\label{convex-forward-backward-lemma}
 Suppose that $f(x) = g(x) + h(x)$ with $g(x)$ an $L$-smooth convex function and $h(x)$ a convex function. Then, setting
 \begin{equation}\label{forward-backward-step}
 x_{n+1} = \text{prox}_{\gamma_n,h}(y_n - \gamma_n\nabla g(y_n)),
\end{equation}
or, written another way, $x_{n+1} = y_n - \gamma_ng_n$ with
\begin{equation}
 g_n = \frac{1}{\gamma_n}(y_n - \text{prox}_{\gamma_n,h}(y_n - \gamma_n\nabla g(y_n))),
\end{equation}
we guarantee that condition \eqref{non-smooth-convex-descent-condition} holds as long as $\gamma_n \leq \frac{1}{L}$.
\end{lemma}
\begin{proof}
 Note that from equation \eqref{forward-backward-step}, we see that
 \begin{equation}
  g_n = \frac{1}{\gamma_n}(y_n - x_{n+1}) = (\nabla g(y_n) + d_n),
 \end{equation}
 where $d_n\in \partial h(x_{n+1})$. Further, the convexity of $g$ and $h$ implies that for any $z$ we have
 \begin{equation}\label{eq-870}
  g(y_n) - g(z) \leq \langle \nabla g(y_n), y_n - z\rangle,
 \end{equation}
 and
 \begin{equation}\label{eq-874}
  h(x_{n+1}) - h(z) \leq \langle d_n, x_{n+1} - z\rangle.
 \end{equation}
 In addition, the $L$-smoothness of $g$ implies that
 \begin{equation}\label{eq-878}
  g(x_{n+1}) - g(y_n) \leq \langle \nabla g(y_n), x_{n+1} - y_n\rangle + \frac{L}{2}\|x_{n+1} - y_n\|^2.
 \end{equation}
 Adding together equations \eqref{eq-870}, \eqref{eq-874}, and \eqref{eq-878}, we get (using that $y_n - x_{n+1} = \gamma_ng_n$)
 \begin{equation}
 \begin{split}
  f(x_{n+1}) - f(y_n) &\leq \langle g_n, y_n - z\rangle - \gamma_n\langle d_n + \nabla g(y_n), g_n\rangle + \frac{L\gamma_n^2}{2}\|g_n\|^2 \\
  & = \langle g_n, y_n - z\rangle + \gamma_n\left(\frac{L\gamma_n}{2} - 1\right)\|g_n\|^2.
  \end{split}
 \end{equation}
 Finally, setting $\gamma_n \leq \frac{1}{L}$ guarantees that $\frac{L\gamma_n}{2} - 1 \leq -\frac{1}{2}$, which completes the proof.

\end{proof}

We proceed to analyze the following accelerated scheme:
\begin{equation}\label{modified-convex-accelerated-scheme}
 x_0 = y_0,~x_{n+1} = y_n - \gamma_ng_n,~y_{n+1} = x_{n+1} + \alpha_n(x_{n+1} - x_n),
\end{equation}
where $g_n$ and $\gamma_n$ are chosen so that the descent condition \eqref{non-smooth-convex-descent-condition} is satisfied. We have the following convergence result, analogous to Theorem \ref{thesis-theorem}.

\begin{theorem}\label{non-smooth-thesis-theorem}
 Let $q_n$ be a sequence of non-negative real number satisfying $q_0 = 0$ and 
 \begin{equation}\label{q-condition-2}
 (q_{n+1} + 1)^2 \leq (q_n + 2)^2 + 1.
 \end{equation}
 Then, if we set $\alpha_n = \frac{q_n}{2 + q_{n+1}}$ in \eqref{modified-convex-accelerated-scheme}, and choose $\gamma_n$ and $g_n$ so that $\gamma_n$ is  decreasing and condition \eqref{non-smooth-convex-descent-condition} is satisfied, we have
 \begin{equation}
  f(x_n) - f(x^*) \leq \frac{2}{\gamma_nq_n(q_n+2)}\|x_0 - x^*\|^2.
 \end{equation}

\end{theorem}
Note that Theorem \ref{non-smooth-thesis-theorem} holds as long as we can ensure the decrease condition \eqref{non-smooth-convex-descent-condition} and that $\gamma_n$ is decreasing. In particular, we do not necessarily need to know any smoothness parameters explicitly.
\begin{proof}
 We largely follow the argument in \cite{siegel2018accelerated} with minor modifications.
 Consider the Lyapunov function
\begin{equation}
    L_n = \gamma_nq_n(q_n + 2)(f(x_n) - f(x^*)) + \frac{1}{2}\|2(y_n - x^*) + q_n(y_n - x_n)\|^2.
\end{equation}
We will show that $L_n$ is decreasing, i.e. that $L_{n+1} \leq L_n$. 

This proves the theorem since $\gamma_nq_n(q_n + 2)(f(x_n) - f(x^*)) \leq L_n$
and $L_0 = 2\|y_0 - x^*\|^2 = 2\|x_0 - x^*\|^2$. 

We begin by breaking $L_n$ into two pieces, namely
\begin{equation}
    L^1_n = \gamma_nq_n(q_n + 2)(f(x_n) - f(x^*))
\end{equation}
and
\begin{equation}
    L^2_n = \frac{1}{2}\|2(y_n - x^*) + q_n(y_n - x_n)\|^2.
\end{equation}
Then we see that
\begin{equation}\label{eq-930}
\begin{split}
    L^1_{n+1} - L^1_n &= \gamma_nq_n(q_n + 2)(f(x_{n+1}) - f(x_n))\\ 
    & +(\gamma_{n+1}q_{n+1}(q_{n+1} + 2) - \gamma_nq_n(q_n + 2))(f(x_{n+1}) - f(x^*)).
\end{split}
\end{equation}
Since by assumption $\gamma_{n+1} \leq \gamma_n$ and $q_{n+1}(q_{n+1} + 2) = (q_{n+1} + 1)^2 - 1 \leq (q_n + 2)^2$ we see that
$\gamma_{n+1}q_{n+1}(q_{n+1} + 2) \leq \gamma_{n}(q_{n} + 2)^2$ and the bottom line in equation \eqref{eq-930} is bounded by
\begin{equation}
     (\gamma_{n}(q_{n} + 2)^2 - \gamma_nq_n(q_n + 2))(f(x_{n+1}) - f(x^*)) = 2\gamma_n(q_n + 2)(f(x_{n+1}) - f(x^*)).
\end{equation}
Thus we get the bound
\begin{equation}\label{eq-942}
    J^1_{n+1} - J^1_n \leq \gamma_n(q_n + 2)[2(f(x_{n+1}) - f(x^*)) + q_n(f(x_{n+1}) - f(x_n))]
\end{equation}
We now utilize the decrease condition \eqref{non-smooth-convex-descent-condition} with $z = x^*$ and with $z = x_n$ in equation \eqref{eq-942} to get
\begin{equation}\label{eq-949}
\begin{split}
    J^1_{n+1} - J^1_n &\leq \langle\gamma_n(q_n + 2)g_n,2(y_n - x^*) + q_n(y_n - x_n)\rangle\\ &-\frac{(\gamma_n(q_n + 2))^2}{2}\|g_n\|^2 
\end{split}
\end{equation}
Next, we consider $J^2_{n+1} - J^2_n$. Note that $J^2_n = (1/2)\|t_n\|^2$ with $$t_n = 2(y_n - x^*) + q_n(y_n - x_n).$$
Thus 
\begin{equation}\label{eq-956}
J^2_{n+1} - J^2_n = \langle t_{n+1} - t_n, t_n\rangle + \frac{1}{2}\|(t_{n+1} - t_n)\|^2.
\end{equation}
So we compute
\begin{equation}
    t_{n+1} - t_n = (2+q_n)(y_{n+1} - y_n)  - q_n(x_{n+1} - x_n) + (q_{n+1} - q_n)(y_{n+1} - x_{n+1}).
\end{equation}
Using the iteration \eqref{modified-convex-accelerated-scheme}, we see that $$y_{n+1} - x_{n+1} = \alpha_n(x_{n+1} - x_n)$$ and $$y_{n+1} - y_n = -\gamma_ng_n + \alpha_n(x_{n+1} - x_n).$$ 
This simplifies to
\begin{equation}
    t_{n+1} - t_n = -(q_n + 2)\gamma_ng_n + (2\alpha_n + q_{n+1}\alpha_n - q_n)(x_{n+1} - x_n) = -(q_n + 2)\gamma_ng_n,
\end{equation}
where the second equality is die to the choice $\alpha_n = \frac{q_n}{2+q_{n+1}}$. Using equation \eqref{eq-956}, we then get
\begin{equation}
 J^2_{n+1} - J^2_n = -\langle (q_n + 2)\gamma_ng_n, 2(y_n - x^*) + q_n(y_n - x_n)\rangle + \frac{(\gamma_n(q_n + 2))^2}{2}\|g_n\|^2.
\end{equation}
 Adding this to equation \eqref{eq-949}, we finally get
 \begin{equation}
  J_{n+1} - J_n = J^1_{n+1} - J^1_n + J^2_{n+1} - J^2_n \leq 0,
 \end{equation}
 as desired.

\end{proof}

Combining Theorem \ref{non-smooth-thesis-theorem} and Lemma \ref{convex-forward-backward-lemma}, we obtain the convergence rate of accelerated forward-backward gradient descent:
\begin{equation}\label{convex-accelerated-forward-backward-it}
 x_0 = y_0,~x_{n+1} = \text{prox}_{\gamma_n,h}(y_n - \gamma_n\nabla g (y_n)),~y_{n+1} = x_{n+1} + \alpha_n(x_{n+1} - x_n),
\end{equation}
in the convex case. 

In particular, if we set  $\gamma_n = \frac{1}{L}$ and $q_n = n$, so that $\alpha_n = \frac{n}{n+3}$ in \eqref{convex-accelerated-forward-backward-it}, then we have
\begin{equation}
 f(x_n) - f(x^*) \leq \frac{2L}{n(n+2)}\|x_0 - x^*\|^2,
\end{equation}
where $f(x) = g(x) + h(x)$ with $g$ convex and $L$-smooth. 

Note that we do not need to know the smoothness parameter $L$ to attain this convergence rate. In particular, by the remark after Theorem \ref{non-smooth-thesis-theorem}, it suffices to choose $\gamma_n$ adaptively to satisfy \eqref{non-smooth-convex-descent-condition} and such that $\gamma_n \leq \gamma_{n-1}$. This can typically be done using a simple line search.

Further, if we wish, a slightly better convergence rate can be obtained by enforcing equality in \eqref{q-condition-2}, which results in a slightly different choice of $\alpha_n$.

\section{Conclusion}
We found it remarkable that so many accelerated methods in the literature were discretizing the same underlying differential equations.
Furthermore, these differential equations have also been considered by the physics community in the context of, for example,
electronic structure calculations \cite{tassone1994acceleration}.

We hope that the ideas we have developed will help lead to the discovery of novel accelerated methods. In the future, we hope to
use our general framework to derive and numerically test specialized accelerated first-order algorithms. Of particular interest are accelerated
methods on manifolds. We believe the differential equation approach will prove important in understanding whether acceleration is
possible in the presence of curvature.

Finally, we believe that our approach simplifies and clarifies the connections between the vast number of accelerated optimization methods
in the literature and hope that it will help other researchers gain intuition about how they work and how to derive new ones.

\section{Acknowledgements}
We would like to thank Professors Russel Caflisch, Stanley Osher, and Jinchao Xu for their helpful suggestions and comments. 
This work was partially supported by AFOSR grant FA9550-15-1-0073.

\bibliographystyle{spmpsci_unsrt}
\bibliography{refs}

\end{document}